\documentclass{article}
\usepackage{amsmath,amssymb,amsthm,times,url}
\newtheorem{theorem}{Theorem}[section]
\newtheorem{lemma}[theorem]{Lemma}
\newtheorem{conjecture}[theorem]{Conjecture}
\newtheorem{definition}[theorem]{Definition}

\newtheorem{remark}[theorem]{Remark}
\newtheorem{corollary}[theorem]{Corollary}
\newtheorem{proposition}[theorem]{Proposition}

\def \be {\begin{equation}}
\def \ee {\end{equation}}
\def \btb {\begin{tablen}}
	\def \etb {\end{tablen}}
\def \bea {\begin{eqnarray}}
\def \eea {\end{eqnarray}}
\def \bean {\begin{eqnarray*}}
	\def \eean {\end{eqnarray*}}
\def \bd {\begin{definition}}
	\def \ed {\end{definition}}
\def \bl {\begin{lemma}}
	\def \el {\end{lemma}}
\def \bcon {\begin{conjecture}}
	\def \econ {\end{conjecture}}
\def \bp {\begin{proposition}}
	\def \ep {\end{proposition}}
\def  \br {\begin{remark}}
	\def \er {\end{remark}}
\def \bt {\begin{theorem}}
	\def \et {\end{theorem}}
\def \bc {\begin{corollary}}
	\def \ec {\end{corollary}}
\def \been {\begin{enumerate}}
	\def \een {\end{enumerate}}
\def \bit {\begin{itemize}}
	\def \eit {\end{itemize}}

\def \nb {\mathbb{N}}

\def \ib {\mathbb{I}}

\def \dq {\mathcal{D}_q}
\def \sq {\mathcal{S}_q}

\newcommand{\qhypergeomphi}[5]{\mbox{$
		_#1 \phi_#2 \left(
		\begin{array}{c}
		\multicolumn{1}{c}{\begin{array}{c} #3
			\end{array}}\\[1mm]
		\multicolumn{1}{c}{\begin{array}{c} #4
			\end{array}}\end{array}
		; \displaystyle{#5}\right) $} }

\usepackage{color}

\begin{document}
	
	\title{A Characterization of Askey-Wilson polynomials}
	\author{Maurice Kenfack Nangho$^{1,2}$ and Kerstin Jordaan$^{3}$\\[2.5pt]
		$^{1}$ Department of Mathematics and Applied Mathematics, University of Pretoria\\
	 $^{2}$ Department of Mathematics and Computer Science, \\University of Dschang, Cameroon\\ {\tt maurice.kenfack@univ-dschang.org}\\[2.5pt]
		$^{3}$ Department of Decision Sciences,\\
		University of South Africa, PO Box 392, Pretoria, 0003, South Africa\\ 
		{\tt jordakh@unisa.ac.za}}
	\maketitle
	
	\begin{abstract}
		\noindent  We show that the only monic orthogonal polynomials $\{P_n\}_{n=0}^{\infty}$ that satisfy
		$$\pi(x)\mathcal{D}_{q}^2P_{n}(x)=\sum_{j=-2}^{2}a_{n,n+j}P_{n+j}(x),\; x=\cos\theta,\;~ a_{n,n-2}\neq 0,~ n=2,3,\dots,$$
		where $\pi(x)$ is a polynomial of degree at most $4$ and $\mathcal{D}_{q}$ is the Askey-Wilson operator, are Askey-Wilson polynomials and their special or limiting cases. This  completes and proves a conjecture by Ismail concerning a structure relation satisfied by Askey-Wilson polynomials. We use the structure relation to derive upper bounds for the smallest zero and lower bounds for the largest zero of Askey-Wilson polynomials and their special cases.
	\end{abstract}

\maketitle
\section{Introduction}
A sequence of polynomials $\{p_n\}_{n=0}^{\infty}$, deg$(p_n)=n$,  is orthogonal with respect to a positive measure $\mu$ on the real numbers $\mathbb{R}$, if
\[\int_{S}p_m(x)p_n(x)d\mu(x)=d_n\delta_{m,n},\;m,n\in \mathbb{N},\]
where $S$ is the support of $\mu$, $d_n>0$ and $\delta_{m,n}$ the Kronecker delta.
A sequence $\{P_n\}_{n=0}^{\infty}$ of monic polynomials orthogonal with respect to a positive measure satisfies a three-term recurrence relation
\begin{equation}\label{e63}
P_{n+1}=(x-a_n)P_{n}-b_nP_{n-1}, \quad n=0,1,2,\ldots\end{equation}
with initial conditions $P_{-1}\equiv  0$, $P_{0}\equiv1$
(note that with this choice of $P_{-1}$, the initial value of $b_0$ is irrelevant)
and recurrence coefficients $a_n\in\mathbb{R}$, $n=0,1,2\dots$, $b_{n}>0$, $n=1,2,\dots.$

 A  sequence of monic orthogonal polynomials is classical if the sequence $\{P_n\}_{n=0}^{\infty}$ as well as $D^mP_{n+m}$, $m\in \mathbb{N}$, where $D$ is the usual derivative $\tfrac{d}{dx}$ or one of its extensions (difference, $q$-difference or divided-difference operator) satisfies a three-term recurrence of the form (\ref{e63}).
 When $D=\frac{d}{dx}$, Hahn \cite{Hahn1936} showed that a sequence of monic orthogonal polynomials $\{P_n(x)\}_{n=0}^{\infty}$
 satisfying
 \begin{equation*}
 \frac{1}{n+1}\frac{dP_{n+1}}{dx}(x)=(x-a'_n)\frac{1}{n}\frac{dP_n}{dx}(x)-\frac{b'_n}{n-1}\frac{dP_{n-1}}{dx}(x),\;\;a'_n,\,b'_n\in \mathbb{R},b'_n\neq 0,
 \end{equation*} satisfies a second order Sturm-Liouville differential equation of the form
 \begin{equation} \label{Sturmleq} \phi(x)\frac{d^2}{dx^2}P_n(x)+\psi(x)\frac{d}{dx}P_n(x)+\lambda_n\,P_n=0.\end{equation}
 where, $\phi$ and $\psi$ are polynomials independant of $n$ with deg($\phi)\leq 2$ and deg($\psi)=1$ while $\lambda_n$ is a constant dependant on $n$. Bochner \cite{bochner1929} first considered sequences of polynomials satisfying (\ref{Sturmleq}) and showed that the orthogonal polynomial solutions of (\ref{Sturmleq}) are Jacobi, Laguerre and Hermite  polynomials, a result known as Bochner's theorem. Bochner's theorem has been generalized and used to characterize Askey-Wilson polynomials (cf. \cite{Mourad2003}). See also \cite{Grunbaum,VinetZhedanov}.

A related problem, due to Askey (cf. \cite{Al-Chihara1972}), is to characterize the orthogonal polynomials whose derivatives satisfy a structural relation of the form
\begin{equation*}
\pi(x)\frac{d}{dx}P_n(x)=\sum_{j=-r}^{s}a_{n,n+j}P_{n+j}(x),\quad n=1,2,\dots
\end{equation*}
and this problem was considered by Maroni (cf. \cite{maroni1985}, \cite{maroni1987}) who called
such orthogonal polynomial sequences semi-classical.

Al-Salam and Chihara \cite{Al-Chihara1972} characterized  {Jacobi}, {Laguerre} and {Hermite} as the only orthogonal polynomials with a structure relation of form
\begin{equation}\label{e57}
\pi(x)\frac{d}{dx}P_n(x)=\sum_{j=-1}^{1}a_{n,n+j}P_{n+j}(x),\quad n=1,2,\dots
\end{equation}
where $\pi(x)$ is a polynomial of degree at most two.
Replacing the usual derivative in (\ref{e57}) by the forward difference operator $$\Delta f(s)=f(s+1)-f(s),$$ Garc\'ia, Marcell\'an and Salto \cite{GaMar1995} proved that {Hahn}, Krawtchouk, {Meixner} and {Charlier} polynomials are the only orthogonal polynomial sequences satisfying

\[
\pi(x)\Delta P_n(x)=\sum_{j=-1}^{1}a_{n,n+j}P_{n+j}(x),\quad n=1,2,\dots
\] with $\pi(x)$ a polynomial of degree two or less.
More recently, replacing the derivative in (\ref{e57}) by the Hahn operator (cf. \cite[(11.4.1)]{Mourad-2005}, \cite{Hahn1949}), also known as the $q$-difference operator or Jackson derivative \cite{Jackson},
\[(D_qf)(x)=\frac{f(x)-f(qx)}{(1-q)x},\]  Datta and Griffin \cite{datta} characterized the big $q$-Jacobi polynomial or one of its special or limiting cases ({Al-Salam-Carlitz 1}, {little and big $q$-Laguerre}, {little $q$-Jacobi}, and {$q$-Bessel} polynomials) as the only orthogonal polynomials that satisfy
\begin{equation}\label{e57a}
\pi(x) D_qP_n(x)=\sum_{j=-1}^{1}a_{n,n+j}P_{n+j}, \quad n=1,2,\dots
\end{equation}
where $\pi(x)$ is a polynomial of degree at most two.

The polynomials mentioned above are all special or limiting cases  of the {Askey-Wilson polynomials} \cite[(1.15)]{Askey-1985}, \cite[(14.1.1)]{KSL}

\begin{equation}\label{e58}
\frac{a^np_n(x;a,b,c,d|q)}{(ab,ac,ad;q)_n}=  \qhypergeomphi{4}{3}{q^{-n},\,abcdq^{n-1},\,ae^{-i\theta},ae^{i\theta}}{ab,\,ac,\,ad}{q,q},\;x=\cos\theta,
\end{equation}
with the multiple $q$-shifted factorials defined by
$(a_1,\dots,a_i;q)_k=\displaystyle{\prod_{j=1}^{i}(a_j;q)_k}$ where the $q$-shifted factorials are given by $(a;q)_0=1, \quad (a;q)_k=\displaystyle{\prod_{j=0}^{k-1}\left(1-aq^j\right)},\\~k=1,2,\dots ~\text{or}~\infty $ and $$\qhypergeomphi{{s+1}}{s}{a_1,\dots ,a_{s+1}}{b_1,\dots,b_s}{q,z}=\sum_{k=0}^{\infty}\frac{(a_1,\dots,a_{s+1};q)_k}{(b_1,\dots,b_s;q)_k}\frac{z^k}{(q;q)_k}.$$
Askey-Wilson polynomials do not satisfy either (\ref{e57}) or (\ref{e57a}) but they do satisfy the shift relation (cf. \cite[(14.1.9)]{KSL})
\begin{eqnarray*}\mathcal{D}_qp_n(x,a,b,c,d|q)=\frac{2q^{\frac{1-n}{2}}(1-q^n)(1-abcdq^{n-1})}{1-q}p_{n-1}(x;aq^{\frac{1}{2}},bq^{\frac{1}{2}} cq^{\frac{1}{2}}, dq^{\frac{1}{2}}|q)\end{eqnarray*}
where {$\mathcal{D}_q$} is the {Askey-Wilson divided difference operator} (cf. \cite[p.35]{Askey-1985}, \cite[(1.16.4)]{KSL}, \cite[(12.1.12)]{Mourad-2005})
\begin{equation}\label{e0.1}
\mathcal{D}_qf(x)=\frac{\breve{f}(q^{\frac{1}{2}}e^{i\theta})-\breve{f}(q^{-\frac{1}{2}}e^{{i\theta}})}{(e^{i\theta}-e^{-i\theta})(q^{\frac{1}{2}}-q^{-\frac{1}{2}})/2},~~\breve f(z)=f\left(\frac{z+z^{-1}}{2}\right), \quad z=e^{\pm i\theta}.\end{equation}
The Askey problem involving the Askey-Wilson operator $\mathcal{D}_q$ is still open but in 2005, Ismail \cite{Mourad-2005}  gave an important hint to the solution of this problem with the following conjecture.
\begin{conjecture}\cite[Conjecture 24.7.9]{Mourad-2005}\label{Mconjecture}
	Let $\{P_n\}$ be orthogonal polynomials and $\pi$ be a polynomial of degree at most 4. Then $\{P_n(x)\}$ satisfies
	\[\pi(x)\mathcal{D}_q^2P_n(x)=\sum_{j=-r}^{s}a_{n,n+j}P_{n+j}(x)\]
	if and only if $\{P_n(x)\}$ are Askey-Wilson polynomials or special cases of them.
\end{conjecture}
The aim of this paper is to complete and prove this conjecture in \S \ref{2} and to apply the explicit structure relation that characterizes Askey-Wilson polynomials to obtain inequalities satisfied by the extreme zeros of these polynomials in \S \ref{3}.
\section{Preliminaries} 
Before moving to our main result let us recall some basic results.
Taking $e^{i\theta}=q^s$, the operator (\ref{e0.1}) reads
\begin{equation*}
\mathcal{D}_qf(x(s))=\frac{f(x(s+\tfrac 12))-f(x(s-\tfrac{1}{2}))}{x(s+\tfrac{1}{2})-x(s-\tfrac{1}{2})},\;\;\;x(s)=\frac{q^{-s}+q^s}{2}.
\end{equation*}
Moreover, $x(s)$ satisfies (cf. \cite{ARS1995})
\bea
\nonumber x(s+n)-x(s)&=&\gamma_n\,\left(x\left(s+\tfrac12n+\tfrac 12\right)-x\left(s+\tfrac12n-\tfrac 12\right)\right),\\
\label{ee2}x(s+n)+x(s)&=&2\alpha_n\,x\left(s+\tfrac 12 n\right),\eea for
$n=0,\,1,\,\dots,$ with the sequences $(\alpha_n),\;(\gamma_n)$ given explicitly by
\be\label{e100}
2\alpha_n=q^{\tfrac n2}+q^{-\tfrac n2},~~ (q^{\tfrac 12}-q^{-\tfrac 12})\gamma_n=q^{\tfrac n2}-q^{-\tfrac n2},~~{\alpha_1}=
\alpha 
\ee
The following hold (cf. \cite[p.302]{Mourad-2005}, \cite[p.169]{foupouagnigni2008})
\bea
\dq(fg)&=&\mathcal{S}_q(f)\mathcal{D}_q(g)+\mathcal{D}_q(f)\mathcal{S}_q(g)\label{e16}
\\
\sq(fg)&=&\sq(f)\sq(g)+U_2\dq(f)\dq(g)\label{e16a}\\
\dq\,\sq &=&\alpha\,\sq\,\dq+U_1\,\dq^2\label{e21a}\\
\sq^2&=& U_1\,\sq\,\dq+\alpha\,U_2\,\dq^2+\ib,\label{e16b}
\eea
where $U_1(x)=(\alpha^2-1)x$, $U_2(x)=(\alpha^2-1)(x^2-1)$, $\ib(f)=f$ and $\sq$ is the averaging operator \cite[(12.1.21)]{Mourad-2005}
$$
\sq\,f(x(s))=\tfrac 12\left(f(x(s+\tfrac{1}{2}))+f(x(s-\tfrac{1}{2}))\right).
$$ 
Unless otherwise indicated, $0<q<1$.  
\section{Proving the conjecture due to Ismail}\label{2}
\medskip We begin by proving a lemma that generalizes a result proved by Hahn in \cite{Hahn1936}. We will denote a monic orthogonal polynomial of precise degree $n$, $n=1,2,\dots$ by $P_n(x)$ which implies that $\frac{1}{\gamma_n}\dq P_n(x)$ will be monic. To see this, normalise the basis in \cite[(20.3.9)]{Mourad-2005}, to obtain the monic polynomial base $\{F_k(x)\}$ where \newline $\displaystyle{F_k(x)=\tfrac{q^{-\tfrac{k^2}{4}}}{(-2)^k}(q^{\tfrac{1}{4}}q^s,q^{\tfrac{1}{4}}q^{-s};q^{\tfrac{1}{2}})_k={ \prod_{j=0}^{k-1}}[x-\zeta_j]}$, for $k=0,1,...$, $x=\cos\theta$ with  $\zeta_j=\tfrac{1}{2}(q^{-\tfrac{1}{4}-\tfrac{j}{2}}+q^{\tfrac{1}{4}+\tfrac{j}{2}})$. It follows from \cite[Thm 2.1]{Mourad2003b} that $P_n(x)=F_n(x)+...$ and, since $\dq F_k(x)=\gamma_kF_{k-1}(x)$ (cf. \cite[20.3.11]{Mourad-2005}), $\dq P_n(x)=\gamma_nF_{n-1}(x)+...$. 
\begin{lemma}\label{Prop1} Let $\{P_n\}_{n=0}^{\infty}$ a sequence of monic orthogonal polynomials. If there are two sequences $(a'_n)$ and $(b'_n)$ such that
	
	\begin{equation}\label{e18}
	\frac{1}{\gamma_{n+1}}\mathcal{D}_q P_{n+1}(x)=(x-a'_n)\frac{1}{\gamma_{n}}\mathcal{D}_q P_{n}(x)-\frac{b'_n}{\gamma_{n-1}}\mathcal{D}_q P_{n-1}(x)+c_n,\; c_n \in \mathbb{R},
	\end{equation}
	then there are two polynomials $\phi(x)$ and $\psi(x)$ of degree at most two and of degree one respectively and a sequence $\{\lambda_n\}_{n=0}^{\infty}$ depending on $n$ such that $P_n(x)$ satisfies the divided difference equation
	{\begin{equation}\label{e18a}
	\phi(x) \mathcal{D}_q^2P_n(x)+\psi(x)\mathcal{S}_q\mathcal{D}_q P_n(x)+\lambda_n P_n(x)=0,\;n\geq 5.
	\end{equation}}
\end{lemma}
\begin{proof} Since $\{P_n\}_{n=0}^{\infty}$ is monic and orthogonal, there exist sequences $\{a_n\}_{n=0}^{\infty}$ and $\{b_n\}_{n=1}^{\infty}$ such that the recurrence relation \eqref{e63} is satisfied. If $f(x)=x-a_n$, it follows from \eqref{ee2} and \eqref{e100} that \begin{equation}\label{S}\sq\,f(x)=\alpha x-a_n.\end{equation}
	Applying the operator $\dq$ to both sides of (\ref{e63}) and using the product rule (\ref{e16}) together with \eqref{S}, yields
	\begin{equation}\label{e17}
	\dq P_{n+1}(x)=(\alpha x-a_n)\dq P_{n}(x)+\sq P_{n}(x)-b_n\dq P_{n-1}(x).
	\end{equation}
	If we apply $\sq$ to both sides of (\ref{e18}) and (\ref{e17}), and use the products (\ref{e16a}) and (\ref{e16b}),
	we obtain respectively
	\begin{subequations}\label{e00}
	\begin{align}
		\frac{1}{\gamma_{n+1}}\sq\dq P_{n+1}(x) &= \left(\alpha x -a'_n\right)\frac{1}{\gamma_n}\sq\dq P_{n}(x)+\frac{1}{\gamma_n}U_2(x)\dq^2P_{n}(x)\nonumber\\
	&-\frac{b'_n}{\gamma_{n-1}}\sq\dq P_{n-1}(x)+c_n.\label{e21}\\
	\sq\dq P_{n+1}(x)&= \left(\alpha^2x+U_1(x)-a_n\right)\sq\dq P_{n}(x)+2\alpha\,U_2(x)\dq^2P_{n}(x)\nonumber\\
	&+P_{n}(x)-b_n\sq\dq P_{n-1}(x)\label{e20} 
	\end{align}
	\end{subequations}
	Applying $\dq$ to both sides of (\ref{e18}) and (\ref{e17}) and then using \eqref{e16} and  (\ref{e21a})
	we obtain respectively
	\begin{subequations}\label{e000}
		\begin{equation}\frac {1}{\gamma_{n+1}}\dq^2P_{n+1}(x) =\frac {\left( \alpha x-a'_n\right)}{\gamma_n}\dq^2P_{n}(x)+\frac {1}{\gamma_n}\sq\dq P_{n}(x)-\frac {b'_n}{\gamma_{n-1}}\dq^2P_{n-1}(x). \label{e23}
	\end{equation}
	\begin{equation}
	\dq^2P_{n+1}(x) = \left( \alpha^2x+U_1(x)-a_n\right)\dq^2P_{n}(x)+2\alpha\sq\dq P_{n}(x)-b_n\dq^2P_{n-1}(x),\label{e22} \end{equation}
	\end{subequations}
	Eliminating $\sq\dq P_{n-1}(x)$ in the system  (\ref{e00}), by subtracting $b_n$ times (\ref{e21}) from  $\frac{b'_n}{\gamma_{n-1}}$ times (\ref{e20}), we have
\begin{equation}\label{e19}
	 A_n\sq\dq P_{n+1}(x)= D_nU_2(x)\dq^2P_{n}(x)+\frac {b'_n}{\gamma_{n-1}}P_{n}(x)-b_nc_n+B_n(x)\sq\dq P_{n}(x)\end{equation} where
	$\displaystyle{ A_n= \frac {b'_n}{\gamma_{n-1}}-\frac {b_n}{ \gamma_{n+1}}}$,  $ B_n(x)=\left (\frac {\alpha^2b'_n}{\gamma_{n-1}}-\frac {\alpha\,b_n}{\gamma_n}\right)x+\frac {b'_n}{\gamma_{n-1}}U_1(x)+\frac {b_na'_n}{\gamma_n}-\frac {b'_na_n}{\gamma_{n-1}}$ and $\displaystyle{D_n=\left(\frac {2\alpha b'_n}{\gamma_{n-1}}-\frac {b_n}{\gamma_n}\right)}$. Eliminating $\sq\dq P_{n+1}(x)$ in \eqref{e00}, by subtracting $\frac{1}{\gamma_{n+1}}$ times (\ref{e20}) from (\ref{e21}), using the relation $\gamma_{n+1}=\alpha_{n}+\alpha\gamma_n$ obtained by direct computation from (\ref{e100}) and substituting $n$ by $n+1$, yields
	\begin{equation}
	\frac{P_{n+1}(x)}{\gamma_{n+2}} =C_n(x)
	\sq\dq P_{n+1}(x) -E_n U_2(x)\dq^2P_{n+1}(x)
	 -A_{n+1}\sq\dq P_{n}(x)+c_{n+1}\label{e24}\end{equation}
 where  $C_n(x)=\displaystyle{\frac{\alpha\,\alpha_{n+1}}{\gamma_{n+1} \gamma_{n+2}}x-\frac{U_1(x)}{\gamma_{n+2}}+\frac{a_{n+1}}{\gamma_{n+2}}-\frac{a'_{n+1}}{\gamma_{n+1}}}$ and $\displaystyle{E_n=\left(\frac {2\alpha}{\gamma_{n+2}}-\frac {1}{\gamma_{n+1}}\right)}$.
Subtracting $\frac{b'_n}{\gamma_{n-1}}$ times (\ref{e22}) from $b_n$ times (\ref{e23}) we obtain
\begin{subequations}
	\begin{align}
	 A_n\dq^2P_{n+1}(x)= B_n(x)\dq^2 P_{n}(x) + D_n\sq\dq P_{n}(x)\label{e25}
	\end{align}
Subtracting $\frac{1}{\gamma_{n+1}}$ times (\ref{e22})  from (\ref{e23}), using again the relation $\gamma_{n+1}=\alpha_{n}+\alpha\gamma_n$ and substituting $n$ by $n+1$, yields
	\begin{align}
	E_n\sq\dq P_{n+1}(x)=C_n(x)\dq^2 P_{n+1}(x)-A_{n+1}\dq^2 P_{n}(x).\label{e26}
	\end{align}
	\end{subequations}
 Eliminating $\dq^2 P_{n+1}(x)$ in (\ref{e26}), by substituting (\ref{e25}) into (\ref{e26}), we obtain
 \begin{equation}\label{e27}
	A_nE_n\sq\dq P_{n+1}(x)(C_n(x) B_n(x)-A_nA_{n+1})\dq^2 P_{n}(x)+ C_n(x)D_n\sq\dq P_{n}(x).
		\end{equation} Using (\ref{e19}), we eliminate $\sq\dq P_{n+1} (x)$ from \eqref{e27} to obtain
\begin{align}\label{e28}
	\phi_{n}(x)\dq^2P_{n}(x)+\psi_{n}(x)\sq\dq P_{n}(x) -& E_n\frac{b'_n}{ \gamma_{n-1}}P_{n}(x)= -E_nb_nc_n,
	\end{align}
	where
		\begin{align*}
		\phi_{n}(x) &=C_n(x)B_n(x)-A_nA_{n+1} -E_nD_nU_2(x)  \\
		\psi_{n}(x)&=C_n(x)D_n-B_n(x)E_n.
		\end{align*}
	Similarly, eliminating $\dq^2P_n  (x)$ in (\ref{e26})  by adding $B_n(x)$ times (\ref{e26}) to  $A_{n+1}$  times (\ref{e25}), and then substituting the resulting relation into (\ref{e24}) to eliminate $\sq\dq\,P_n (x)$, yields
\begin{equation} 
	\phi_{n}(x)\dq^2P_{n+1}(x)+\psi_{n}(x)\sq\dq\,P_{n+1}(x)-\label{e29}\frac{ D_n}{ \gamma_{n+2}}P_{n+1}(x)=-D_nc_{n+1},
\end{equation}
where $\phi_{n}(x)$ and $\psi_{n}(x)$ are the polynomial coefficients of (\ref{e28}).
Substituting $U_1(x)=(\alpha^2-1)x$ into (\ref{e22}) and subtracting $\frac{1}{\gamma_{n}}$ times the obtained equation from  $2\alpha$ times (\ref{e23})  to elliminate $\sq\dq P_n(x)$,  yields 
\begin{equation}\label{nttrr}\frac{x}{ \gamma_n }\dq^2P_{n}(x)= E_{n-1}\dq^2P_{n+1}(x)+\frac{( 2\alpha  a_n'-a_n)}{ \gamma_n}\dq^2P_n(x)+ D_n \dq^2P_{n-1}(x).\end{equation} Substituting $\sq\dq P_{n+1}$,   $\sq\dq P_{n}$ and  $\sq\dq P_{n-1}$  obtained from (\ref{e26}), into (\ref{e21}) and repeatedly applying \eqref{nttrr}, we obtain
$c_n=\displaystyle{\sum_{k=-2}^{2}d_{n,k}\dq^2P_{n+k}(x)}, n\geq 2.$ Since $\dq^2P_{j+2}$ is of degree $j$, $\{\dq^2P_{j+2}\}_{j=0}^{\infty}$ forms a basis for the space of polynomials and therefore $c_n=0$ for $n\geq 5.$ In the sequel of this proof, we will assume that $n\geq 5.$

\medskip \noindent Using the relation 
\[\sq f(x(s))=\mathcal{T}_1f(x(s))-\frac{x(s+\frac{1}{2})-x(s-\frac{1}{2})}{2}\dq f(x(s)),\;\;\mathcal{T}_{\nu}f(x(s))=f(x(s+\tfrac{\nu}{2})),\] that follows from the definitions of $\sq$ and $\dq$, in (\ref{e28}) with $n$ replaced by $n+1$ and also in (\ref{e29}), we obtain respectively
\begin{subequations}\begin{equation}\label{e28b}
	\sigma_{n+1}(x(s))\dq^2P_{n+1}(x(s))+\psi_{n+1}(x(s))\mathcal{T}_1\dq P_{n+1}(x(s))
	-\frac{E_{n+1}b'_{n+1}}{ \gamma_{n}}P_{n+1}(x(s))=0,\end{equation}
	\begin{equation}\label{e29b}
	\sigma_{n}(x(s))\dq^2P_{n+1}(x)+\psi_{n}(x(s))\mathcal{T}_1\dq P_{n+1}(x(s))-D_n\frac{1}{ \gamma_{n+2}}P_{n+1}(x(s))=0,
	\end{equation}\end{subequations} where $$\sigma_n(x(s))=\phi_{n}(x(s))-\frac{x(s+\frac{1}{2})-x(s-\frac{1}{2})}{2}\psi_n(x(s)).$$
Subtracting $\sigma_{n+1}(x(s))$ times (\ref{e29b}) from $\sigma_{n}(x(s))$ times (\ref{e28b}), yields
\begin{align}
	&\left(\phi_{n}(x(s))\psi_{n+1}(x(s))-\phi_{n+1}(x(s))\psi_{n}(x(s))\right)\mathcal{T}_1\dq P_{n+1}(x(s))
	+\label{e30a}\\
	&\left(\frac{\sigma_{n+1}(x(s))D_n}{ \gamma_{n+2}}-\frac{\sigma_{n}(x(s))E_{n+1}b'_{n+1}}{ \gamma_n}\right)P_{n+1}(x(s))=0,\nonumber
\end{align} where $\mathcal{T}_1\dq P_{n+1}(x(s))=\frac{P_{n+1}(x(s+1))-P_{n+1}(x(s))}{x(s+1)-x(s)}$ by definition.
\medskip\noindent  Since $P_{n+1}$ is a function of the variable $x=\cos\theta$, its zeros are in the interval $(-1,1)$. Let $-1<x(s_1)<x(s_2)<...<x(s_{n+1})<1$ denote the zeros of $P_{n+1}(x(s))$. For $j=1,2,...,n+1$  there is $\theta_j, 0<\theta_j<\pi$, such that $x(s_j)=\frac{q^{s_j}+q^{-s_j}}{2}=\frac{e^{i\theta_j}+e^{-i\theta_j}}{2}$ and it follows that $x(s_j+1)=\frac{qe^{i\theta_j}+q^{-1}e^{-i\theta_j}}{2}=\frac{(q^2+1)\cos\theta_j+i(q^2-1)\sin\theta_j}{2q}\notin \mathbb{R}$ for $0<q<1$. Therefore $P_{n+1}(x(s_j+1))\neq 0$ and hence $\mathcal{T}_1\dq P_{n+1}(x(s_j))\neq 0$ for $j=1,2,...,n+1$. So,  by (\ref{e30a}), the polynomial $F_n(x(s))=\phi_{n}(x(s))\psi_{n+1}(x(s))-\phi_{n+1}(x(s))\psi_{n}(x(s))$, which is of degree at most 3, will vanish at n+1 zeros of $P_{n+1}$, $n\geq 5$. Hence $F_n(x)$ is equal to zero for all $x$ and there exists $G_n$, $n\in \mathbb{N}$, such that  $\phi_{n+1}(x)=G_n \phi_{n}(x)$ and $\psi_{n+1}(x)=G_n \psi_{n}(x)$. Iterating these relations, we obtain $\phi_{n}(x)=H_n\phi_5(x)$ and $\psi_{n}(x)=H_n\psi_5(x)$,  $H_n=\displaystyle{\prod_{j=5}^{n-1}G_j}$. Finally, dividing both sides of (\ref{e28}) by $H_n$ and keeping in mind that $c_n=0$ for $n\geq 5$, we obtain the result.
\end{proof}

We now state and prove our main result.
\begin{theorem}\label{th1} Let $\{P_n\}_{n=0}^{\infty}$ be a sequence of monic polynomials orthogonal with respect to a positive weight function $w(x)$. The following properties are equivalent.
	\been
	\item[(a)] There is a polynomial  $\pi(x)$ of degree at most $4$ and constants $a_{n,n+k}$, $k\in\{-2,-1,0,1,2\}$ with $a_{n,n-2}\neq 0$ such that $P_n$ satisfies the structure relation
	\begin{equation*}
	\pi(x)\mathcal{D}_q^2P_n(x)=\sum_{k=-2}^{2}a_{n,n+k}P_{n+k}(x), \;\; n=2,3,\dots;
	\end{equation*}
	\item[(b)] There is a polynomial $\pi(x)$ of degree at most four such that $\{\mathcal{D}_q^2P_j\}_{j=2}^{\infty}$ is orthogonal with respect to $\pi(x)\,w(x)$;
	
	\item[(c)] There are two polynomials $\phi(x)$ and $\psi(x)$ of degree at most two and of degree one respectively and a constant $\lambda_n$ such that
	\begin{equation}\label{e10}
	\phi(x) \mathcal{D}_q^2P_n(x)+\psi(x)\mathcal{S}_q\mathcal{D}_q P_n(x)+\lambda_n P_n(x)=0, \;\;n=5,6,\dots.
	\end{equation}
	\een
\end{theorem}

\begin{proof}[Proof of Theorem \ref{th1}] 
	The proof is organized as follows.\\
	Step 1 $(a)\Rightarrow\,(b)\Rightarrow\,(a)$ which is equivalent to $(a)\Leftrightarrow\,(b)$.\\
	Step 2  $(b)\Rightarrow\,(c)\Rightarrow\,(a)$ which, taking into account Step 1, is equivalent to $(b)\Leftrightarrow (c)$.

	\medskip \noindent Step 1: Assume that $(a)$ is satisfied. 
	Let $m,n\in\nb$, $m,n\geq2$ and $m\leq\,n$. From (a), there is a polynomial $\pi(x)$ of degree at most four and there exist constants $a_{n,n+j}$, $j\in\{-2,\,-1,\,0,\,1,\,2\}$ such that
	\begin{equation}\label{e30}
	\pi(x)\dq^2P_n(x)=\sum_{j=-2}^{2}a_{n,n+j}P_{n+j}(x),\,{\rm with}\, a_{n,n-2}\neq 0.
	\end{equation}
	Since $m\leq\,n$ we have that $m-2\leq\,n-2\leq\,n+j\leq\,n+2$ for $j\in\{-2,\,-1,\,0,\,1,\,2\}$. Multiplying both sides of (\ref{e30}) by $w(x)\dq^2P_m(x)$, integrating on $(a,b)$
	and then taking into account the fact that $\{P_j\}_{j=0}^{\infty}$ is orthogonal on the interval $(a,b)$ with respect to the weight function $w(x)$, we obtain
	\[
	\int_{a}^{b}\dq^2P_m(x)\dq^2P_n(x)\pi\,(x)w(x)dx
	\left\{\begin{array}{lll}
	=0&{\rm if}& m<n\\
	\neq 0& {\rm if}& m=n.
	\end{array}
	\right.
	\]
	If $n<m$, interchanging $m$ and $n$ in the above argument yields
	\[\int_{a}^{b}\dq^2P_n(x)\dq^2P_m(x)\pi\,(x)w(x)dx=0.\]
	\\
	Now let $n\in\nb$, $n\geq2$ and assume $(b)$.
	Since $\pi\,(x)\dq^2P_n(x)$ is a polynomial of degree less or equal to $n+2$, it can be expanded in the orthogonal basis $\{P_j\}_{j=0}^{\infty}$ as
	$\displaystyle{\pi\,(x)\dq^2P_n(x)=\sum_{k=0}^{n+2}a_{n,k}P_k(x)}$,
	where, for $k\in\{0,...,n+2\}$,  $a_{n,k}$ is given by
	\[ a_{n,k}\int_{a}^{b}\left(P_k(x)\right)^2w(x)dx=\int_{a}^{b}P_k(x)\dq^2P_n(x)\pi(x)w(x)dx.\]
	Since $\dq^2P_n(x)$ is of degree $n-2$ we deduce from the hypothesis that  $a_{n,k}=0$ for $k\in\{0,...,n-3\}$ and $a_{n,n-2}\neq 0$.\\
	
	\medskip \noindent Step 2: We suppose (b) and we prove (c). Firstly, we prove that polynomials in the sequence  $\{P_n\}_{n=0}^{\infty}$ satisfy an equation of type (\ref{e18}). Let $n\in\nb$, $n\geq 2$ and denote the leading coefficient of $P_n$ by $\gamma_n$, then, since $\frac{x}{ \gamma_n}\dq\,P_n$ is a monic polynomial of degree $n$, it can be expanded as
	\begin{equation}\label{e31a}
	x\frac{1}{ \gamma_n}\dq\,P_n(x)=\frac{1}{ \gamma_{n+1}}\dq\,P_{n+1}(x)+\sum_{j=1}^{n}\frac{e_{n,j}}{ \gamma_j}\dq\,P_j(x),\; e_{n,j}\in \mathbb{R}.
	\end{equation}
     Applying $\dq$ to both sides of (\ref{e31a}) and using (\ref{e16}), we obtain
	\begin{equation}\label{e31}
	\left(\alpha\,x\right)\frac{1}{ \gamma_n}\dq^2\,P_n(x)+\frac{1}{ \gamma_n}\sq\dq\,P_n(x)=\frac{1}{ \gamma_{n+1}}\dq^2\,P_{n+1}(x)+\sum_{j=2}^{n}\frac{e_{n,j}}{ \gamma_j}\dq^2P_j(x).
	\end{equation}
	Substituting $U_1(x)=(\alpha^2-1)x$ into  (\ref{e22}), yields	
	\begin{equation}\label{e32a}
	\dq^2P_{n+1}(x)=[\left(2\alpha^2-1\right)x-a_n]\dq^2P_n(x)+2\alpha\sq\dq\,P_n(x)-b_n\dq^2P_{n-1}(x).
	\end{equation}
	Eliminating $\sq\dq\,P_n(x)$ in (\ref{e31}) by subtracting $\frac{1}{\lambda_n}$ times (\ref{e32a}) from $2\alpha$ times (\ref{e31}), we obtain
	\begin{align}\label{e32}
	\frac{x+a_n}{ \gamma_n}\dq^2\,P_n(x)&+\frac{b_n}{ \gamma_n}\dq^2P_{n-1}(x)\\
	&=\left(\frac{2\alpha}{ \gamma_{n+1}}-\frac{1}{ \gamma_n}\right)\dq^2\,P_{n+1}(x)+\sum_{j=2}^{n}\frac{2\alpha\,e_{n,j}}{ \gamma_j}\dq^2P_j(x)\nonumber.
	\end{align}
	Since $\{\frac{\dq^2\,P_n}{ \gamma_{n}\gamma_{n-1}}\}$ is a family of monic orthogonal polynomials, there are $a''_n$ and $b''_n>0$ such that
	\begin{equation}\label{e34}
	x\frac{\dq^2P_n(x)}{ \gamma_n}=\frac{\gamma_{n-1}}{ \gamma_{n+1}\gamma_{n}}\dq^2P_{n+1}(x)+a''_n\dq^2P_n(x)+b''_n\dq^2P_{n-1}(x).
	\end{equation}
	Substituting (\ref{e34}) into (\ref{e32}) and  using the relation $\gamma_{n+1}-2\alpha\gamma_n+\gamma_{n-1}=0$, obtained by direct computation from \eqref{e100}, we obtain
	\[\left(a''_n+\frac{a_n}{ \gamma_n}\right)\dq^2P_n(x)+\left(b''_n+\frac{b_n}{ \gamma_n}\right)\dq^2P_{n-1}(x)=\sum_{j=2}^{n}\frac{2\alpha e_{n,j}}{ \gamma_j}\dq^2P_j(x).\]
	Therefore,
	$e_{n,j}=0$ for $j\in\{2,3,...n-2\}$ and (\ref{e31a}) can be written as
	\[
	\frac{x}{ \gamma_n}\dq\,P_n(x)=\frac{1}{ \gamma_{n+1}}\dq\,P_{n+1}(x)+\frac{e_{n,n}}{ \gamma_n}\dq\,P_n(x)+\frac{e_{n,n-1}}{ \gamma_{n-1}}\dq\,P_{n-1}(x)+e_{n,1}.
	\]
	The result follows from Lemma \ref{Prop1}.
	\noindent Finally, we prove that $(c)\Rightarrow (a)$.\\
	Adding $\psi(x)$ times (\ref{e20}) to $\phi(x)$ times (\ref{e22}) and then using the assumption (c), we obtain
	\begin{align}\label{nnnn}
	\lambda_{n+1}P_{n+1}(x)=&\lambda_n\left(\alpha^2x+U_1(x)-a_n\right)P_n(x)-2\alpha (\phi(x)\sq\dq\,P_n(x)\\
	&+U_2(x)\psi(x)\dq^2P_n(x))
	-\psi(x)\,P_n(x)-b_n\lambda_{n-1}P_{n-1}(x).\nonumber\end{align}
	Multiplying \eqref{nnnn} by $\psi(x)$ and substituting  $\psi(x)\sq\dq\,P_n(x)=-\phi(x)\dq^2P_n(x)-\lambda_nP_n(x)$ obtained from \eqref{e10} and $U_1(x)=\left(\alpha^2-1\right)x$, yields
	\begin{align*}
		2\alpha\left(\phi^2(x)-U_2(x)\psi^2(x)\right)&\dq^2P_n(x)
		=\lambda_{n+1}\psi(x)\,P_{n+1}(x)+[\psi^2(x)-2\alpha\lambda_n\phi(x)\\
		-&\lambda_n\psi(x)\left((\alpha^2-1)x-a_n\right)]P_n(x)+\lambda_{n-1}b_n\psi\,P_{n-1}(x).\end{align*}
	Taking $\phi\,(x)=\phi_2x^2+\phi_1x+\phi_0$ and $\psi\,(x)=\psi_1x+\psi_0$ and using the three-term recurrence relation (\ref{e63}), we transform the above equation into
	\begin{equation}\label{e34a}
	\left(\phi^2(x)-U_2(x)\psi^2(x)\right)\dq^2P_n(x)=\sum_{j=-2}^{2}a_{n,n+j}P_{n+j}(x),
	\end{equation}
	where
	$2\alpha a_{n,n-2}=
		\psi_1b_{n-1}b_n\left(\psi_1-\lambda_n(2\alpha\phi_2+(\alpha^2-1)+\lambda_{n-1}\right).$ Clearly $a_{n,n-2}\neq 0$ for $b_n>0$, since $\psi_1\neq0$ and $\psi_1$ also does not depend on $n$. This yields the required result.
\end{proof}

\begin{corollary} A sequence of monic orthogonal polynomials satisfies the relation
	\be\label{e11}
	\pi(x)\mathcal{D}_q^2P_n(x)=\sum_{k=-2}^{2}a_{n,n+k}P_{n+k}(x),\;\;a_{n,n-2}\neq 0,~~ x=\cos\theta,
	\ee
	where $\pi$ is a polynomial of degree at most 4, if and only if $P_n(x)$ is a multiple of the Askey-Wilson polynomial for some parameters $a,\,b,\,c,\,d,$ including limiting cases as one or more of the parameters tend to $\infty$.
\end{corollary}
\begin{proof}
Let $\{P_n(x)\}_{n=0}^{\infty},\; x=\cos\,\theta$, be a sequence of monic orthogonal polynomials and $\pi(x)$ be a polynomial of degree at most $4$. It follows from Theorem \ref{th1} that $\{P_n(x)\}$ satisfies (\ref{e11}) if and only if $P_n(x)$ is polynomial solution of (\ref{e10}). 
It was proved in \cite[Thm. 3.1]{Mourad2003} that \eqref{e10} has a polynomial solution of degree $n$ if and only if the solution is up to a multiplactive factor equal to an Askey-Wilson polynomial, a special case or a limiting case of an Askey-Wilson polynomial when one or more of the parameters tend to $\infty$ and these limiting cases are orthogonal \cite[Remark 3.2]{Mourad2003}, which yields the result. 
\end{proof}
\begin{remark} It follows from (\ref{e34a}) and Theorem \ref{th1} that $\{\dq^2P_n\}_{n=2}^{\infty}$ is orthogonal with respect to\\
	$\left(\phi^2(x)-U_2(x)\psi^2(x)\right)w(x)$. So, there is a positive constant $c$ such that
	$	\pi(x) = c\left(\phi^2(x)-U_2(x)\psi^2(x)\right). $
	Without loss of generality, we can take $c=1$ so that
	\begin{equation} \label{ee4}\pi(x)=\phi^2(x)-U_2(x)\psi^2(x).
 \end{equation}
\end{remark}
In the following remark we provide the polynomial coefficients $\phi(x)$ and $\psi(x)$ of (\ref{e10}) as well as the polynomial $\pi(x)$ in (\ref{e11}) for the monic Askey-Wilson polynomials.
\begin{remark}\label{newremark}
Let $a_n:=a_n(a,b,c,d)$ and $b_n:=b_n(a,b,c,d)$ be the coefficients of (\ref{e63}) for the monic Askey-Wilson polynomials $$2^n(abcdq^{n-1};q)_nP_n(x;a,b,c,d|q)=p_n(x;a,b,c,d|q).$$ Since $\mathcal{D}_q P_n(x;a,b,c,d|q)=\gamma_{n}P_{n-1}(x;aq^\frac{1}{2},bq^\frac{1}{2},cq^\frac{1}{2},dq^\frac{1}{2}|q)$, the coefficients of (\ref{e18}) can be deduced from those of (\ref{e63}) as follows
\begin{equation} \label{a}a'_n=a_{n-1}(aq^\frac{1}{2},bq^\frac{1}{2},cq^\frac{1}{2},dq^\frac{1}{2})\hspace{0.5cm} {\rm and} \hspace{0.5cm}b'_n=b_{n-1}(aq^\frac{1}{2},bq^\frac{1}{2},cq^\frac{1}{2},dq^\frac{1}{2}).\end{equation}
It is shown in the proof of Lemma \ref{Prop1} that $\phi(x)$ and $\psi(x)$ in \eqref{e10} are obtained by letting $n=5$ in the polynomial coefficients of (\ref{e28}). Hence, taking $n=5$ in the expressions for $\phi_n(x)$ and $\psi_n(x)$ (cf. (\ref{e28})) and using \eqref{a} together with the three-term recurrence relation for monic Askey-Wilson polynomials (cf. \cite[14.1.5]{KSL}), we obtain, up to a multiplicative factor,
\begin{align} 
\phi(x)=&2(abcd+1)x^2-(abc+abd+acd+bcd+a+b+c+d)x\nonumber\\
& \label{phi}+ab+ca+ad+bc+bd+cd-dcba-1;\\
\psi(x)=&\frac{(abcd-1)4 \sqrt{q}x}{q-1}+\frac{(a+b+c+d-abc-abd-acd-bcd)2 \sqrt {q}}{q-1}.\label{psi}
\end{align}
Substituting the expressions  (\ref{phi}) and (\ref{psi}) for $\phi(x)$ and $\psi(x)$ into (\ref{ee4}) and taking into account the fact that $U_2(x)=(\alpha^2-1)(x^2-1)$, we obtain after simplification,
\[\pi(x)=16abcd(x-\tfrac{a^{-1}+a}{ 2})(x-\tfrac{b^{-1}+b}{ 2})(x-\tfrac{c^{-1}+c}{ 2})(x-\tfrac{d^{-1}+d}{ 2}).\]
\end{remark}
Ismail \cite[Remark 3.2]{Mourad2003} points out that solutions to \eqref{e10} do not necessarily satisfy the orthogonality relation of Askey-Wilson polynomials using the example $\displaystyle{\lim_{d\to \infty}p_n(x;a,b,c,d)}$ to show that the moment problem is indeterminate for $0<q<1$ and max$\{ab,ac,ad\}<1$ while, for $q>1$ and min$\{ab,ac,ad\}>1$, the moment problem is determinate and the polynomials are special Askey-Wilson polynomials. In the next proposition, we explicitly state the various limiting cases for Askey-Wilson polynomials. 

\begin{proposition}\label{lemma3} Let $q>0$, $q \neq 1$. Then, for the Askey-Wilson polynomials $p_n(x;a,b,c,d|q)$, we have 
	\begin{itemize}
		\item[(i)]  
		$\displaystyle{\lim_{d\rightarrow\infty}\frac{p_n(x;a,b,c,d|q)}{(ad;q)_n}=(bc)^nq^{n(n-1)}p_n(x;a^{-1},b^{-1},c^{-1}|q^{-1}),}$
		where \\$p_n(x;a^{-1},b^{-1},c^{-1}|q^{-1})$ denotes continuous dual $q$-Hahn polynomials with the orthogonality relation for $q>1$ given by \cite[(14.4.2)]{KSL}).
		\item[(ii)]
		$\displaystyle{\lim_{c,d\rightarrow\infty}\frac{a^np_n(x;a,b,c,d|q)}{ (ac;q)_n(ad;q)_n}=(-b)^nq^{\frac{n(n-1)}{2}}Q_n(x;a^{-1},b^{-1}|q^{-1}),}$
		where $Q_n$ denotes the Al-Salam-Chihara \newline polynomials with the orthogonality relation for $q>1$ given by \cite[(14.8.2)]{KSL}.
		\item[(iii)] 
		$\displaystyle{\lim_{b,c,d\rightarrow\infty}\frac{a^{n}p_n(x;a,b,c,d|q)}{(ab;q)_n(ac;q)_n(ad;q)_n}=a^{-n}H_n(x;a^{-1}|q^{-1}),}$
		where $H_n$ is the continuous big $q$-Hermite polynomials with the orthogonality relation for $q>1$ given by \cite[(14.8.2)]{KSL} .
		\item[(iv)] 
		$\displaystyle{\lim_{a,b,c,d\rightarrow\infty}\frac{a^{2n}p_n(x;a,b,c,d|q)}{(ab;q)_n(ac;q)_n(ad;q)_n}=H_n(x|q^{-1}),}$
		where $H_n$ denotes the continuous $q$-Hermite polynomials \cite[(14.26.2)]{KSL}.
	\end{itemize}
\end{proposition}
\begin{proof}\vspace{-0.25cm}	\begin{align*}\hspace{-0.75cm}
	\lim_{d\rightarrow\infty}\frac{a^np_n(x;a,b,c,d|q)}{(ab;q)_n (ac;q)_n(ad;q)_n}&=\sum_{k=0}^{n}\frac{(q^{-n};q)_k(bcq^n)^k}{ (ab;q)_k(ac;q)_k(q;q)_k}\prod_{j=0}^{k-1}(1-2aq^jx+a^2q^{2j})\\
	&= \frac{\left( 2\,abc \right)^{n}{q}^{n(n-1)}}{(ab;q)_n (ac;q)_n}\,q_n(x;a,b,c|q),
	\end{align*}
	where $q_n$ is a monic polynomial satisfying the three-term recurrence relation
	\begin{align}\label{TTR1}
	q_{n+1}(x;a,b,c|q)&=(x-\widetilde{a}_n)q_{n}(x;a,b,c|q)-\widetilde{b}_nq_{n-1}(x;a,b,c|q),
	\end{align}
	where	$\widetilde{a}_n={\frac {ab{q}^{n}+ac{q}^{n}+bc{q}^{n}+{q}^{n}q-q-1}{2ac \left( {q}
			^{n} \right) ^{2}b}}$ and 
	$\widetilde{b}_n={\frac { \left( {q}^{n}-1 \right)  \left( bc{q}^{n}-q \right)
			\left( ac{q}^{n}-q \right)  \left( ab{q}^{n}-q \right) }{2{a}^{2}{c}^{
				2} \left( {q}^{n} \right) ^{4}{b}^{2}}}.$
	From (\ref{TTR1}) and  \cite[(14.3.5)]{KSL} , we obtain
$2^nq_n(x;a,b,c|q)=p_n(x;a^{-1},b^{-1},c^{-1}|q^{-1})$
	where $p_n(x;a^{-1},b^{-1},c^{-1}|q^{-1})$ denotes continuous dual $q$-Hahn polynomials \cite[(14.3.1)]{KSL}. Therefore 
$\lim_{d\rightarrow\infty}\frac{p_n(x;a,b,c,d|q)}{(ad;q)_n}=(bc)^nq^{n(n-1)}p_n(x;a^{-1},b^{-1},c^{-1}|q^{-1}).$
	The other limits are obtained in an analogous manner.
\end{proof}
In \cite{Koorn}, Koornwinder obtained another structure relation for Askey -Wilson polynomials in the form $Lp_n=r_np_{n+1}+s_np_{n-1}$, where $L$ is the divided $q$-difference linear operator defined by \cite[(1.8)]{Koorn}. The connection of the structure relation \cite[(4.7)]{Koorn} to \eqref{e11} is provided in the following proposition. 
\begin{proposition} Let $P_n(x)=P_n(x;a,b,c,d|q)=\frac{p_n(x;a,b,c,d|q)}{ 2^n(abcdq^{n-1};q)_n}$ denote the monic Askey-Wilson polynomials. Then, for the operator $L$ defined by \cite[(1.8)]{Koorn} we have that, for $x=\cos \theta$,
	\begin{align*}\psi(x)&(LP_n)(x)
		=\frac{1-q^2}{2q}\pi(x)\mathcal{D}_q^2P_n(x)+\frac{q-1}{\sqrt{q}}\\
		&\times[\psi(x)^2+\frac {4\sqrt {q}(q^n-1)(q^{n-1}abcd-1)}{(q -1)^2 q^{n-1}}(\frac{1}{\sqrt{q}}\phi(x)+\frac{(q-1)^2}{2q}x\psi(x))]P_n(x),\;
		\end{align*}
	where $\phi(x)$ and $\psi(x)$ are the polynomial coefficients of (\ref{e10}) given by (\ref{phi}) and (\ref{psi}).
	\end{proposition}
\begin{proof}
	It follows from \cite[Thm 6]{fkm10} that the structure relation \cite[(4.7)]{Koorn} can be written as \newline
	$(LP_n)(x(s))=\xi\left(2\phi(x(s))\mathcal{D}_q\mathcal{S}_q+2\psi(x(s))\mathcal{S}_q^2-\psi(x(s))\right)P_n(x(s)),$
	where $x(s)=\frac{q^{-s}+q^s}{ 2}$ ($q^s=e^{i\theta}$) and $\xi$ is a constant. Take $n=1$, to obtain, after simplification, $2q\xi=1-q^2$. Use (\ref{e21a}) and (\ref{e16b}) to write $LP_n$  in terms of $\mathcal{D}_q^2$ and $\mathcal{S}_q\mathcal{D}_q$. Now, multiply the relation by $\psi$ and use the fact that Askey-Wilson polynomials satisfy (\ref{e10}) with polynomial coefficients $\phi$ and $\psi$ and the constant $\lambda_n=-4\frac {\sqrt {q}({q}^{n}-1)( {q}^{n}abcd-q)}{(-1+q) ^{2}{q}^{n}}$  given in \cite[(16.3.19) and (16.3.20)]{Mourad-2005}, to obtain the result.\end{proof}
	
	In the following proposition we consider the conditions under which the $n$th degree polynomial $P_n(x)$ in a sequence of polynomials orthogonal with respect to a weight $w(x)$ can be written as a linear combination of the polynomials $\dq^2P_{n+j}(x)$, $j,n \in \nb$. A structure relation of this type involving the forward arithmetic mean operator $\displaystyle{\tfrac{1}{2}(f(s+1)+f(s))}$ is proved in  \cite{CSM}.  
	\begin{proposition}
		Let $\{P_n\}_{n=0}^{\infty}$ be a sequence of monic polynomials orthogonal with respect to a weight function $w(x)$ defined on $(a,b)$. Suppose $\{\dq^2P_j\}_{j=2}^{\infty}$ is a sequence of polynomials orthogonal with respect to the weight function $\pi(x)\,w(x)$ on $(a,b)$ where $\pi(x)$ is a polynomial of degree at most 4. Then for each $n\in\nb$, $n\geq4$, there exist constants $b_{n,n+j}$, $j\in\{-2,\,-1,\,0,\,1,\,2\}$ such that
			\begin{equation}\label{e33}
			P_n(x)=\sum_{j=-2}^{2}b_{n,n+j}\dq^2P_{n+j}(x).
			\end{equation}
	\end{proposition}
	\begin{proof} Let $n\in\nb$, $n\geq4$. 
				Since $\{\dq^2P_j\}_{j=2}^{\infty}$ is orthogonal with respect to a weight function $\pi(x)\,w(x)$ on $(a,b)$, $P_n$ can be expanded in terms of the orthogonal basis as
				$\displaystyle{P_n(x)=\sum_{k=2}^{n+2}b_{n,k}\dq^2P_k(x)}$,
				where, for each fixed $k$, $k\in\{2,3,...,n+2\}$, $b_{n,k}$ is given by
				\[b_{n,k}\int_{a}^{b}\left(\dq^2P_k(x)\right)^2\pi(x)\,w(x)dx=\int_{a}^{b}\,\dq^2P_k(x)P_n(x)\pi(x)w(x)dx.\]
				Since $\pi(x)\,\dq^2P_k(x)$ is a polynomial of degree at most $k+2$ and $\{P_j\}_{j=0}^{\infty}$ is orthogonal with respect to $w(x)$ on $(a,\,b)$, it follows that $b_{n,k}=0$, for $k\in\{2,..,n-3\}$.
	\end{proof}
 \section{Extreme zeros of Askey-Wilson polynomials and special cases}\label{3}
In this section we obtain the explicit structure relation (\ref{e11}) characterizing Askey-Wilson polynomials and then use the relation to derive bounds for the extreme zeros of the Askey-Wilson polynomials and their special cases.  
\begin{lemma}\label{lemma4} The monic Askey-Wilson polynomials $P_n(x;a,b,c,d|q)$ satisfy the following contiguous relations \begin{align*}
			\left(x-\tfrac{a^{-1}+a}{2}\right)P_n(x;aq,b,c,d|q)&=~P_{n+1}(x;a,b,c,d|q)+k_{n}^{(a,b,c,d)}P_n(x;a,b,c,d|q),\\
			\left(x-\tfrac{b^{-1}+b}{2}\right)P_n(x;a,bq,c,d|q)&=~P_{n+1}(x;a,b,c,d|q)+k_{n}^{(b,a,c,d)}P_n(x;a,b,c,d|q),\\
			\left(x-\tfrac{c^{-1}+c}{2}\right)P_n(x;a,b,cq,d|q)&=~P_{n+1}(x;a,b,c,d|q)+k_{n}^{(c,b,a,d)}P_n(x;a,b,c,d|q),\\
			\left(x-\tfrac{d^{-1}+d}{2}\right)P_n(x;a,b,c,dq|q)&=~P_{n+1}(x;a,b,c,d|q)+k_{n}^{(d,b,c,a)}P_n(x;a,b,c,d|q),
		\end{align*}
	with  $\displaystyle{k_n^{(a,b,c,d)}=-{\frac { \left(1- ab{q}^{n} \right)  \left( 1-ac{q}^{n} \right)
			\left(1- ad{q}^{n} \right)  \left(1- abcd{q}^{n-1}\right) }{2 a\left(1- ab
			cd {q}^{2n-1} \right) \left(1- abcd q^{2n}\right) }}}.$
\end{lemma}
\begin{proof} Substitute $P_n(x;a,b,c,d|q)$ into \cite[(2.15)]{Askey-1985} to obtain the first relation. For the others, permute $a$ and $e$, $e\in \{b,c,d\}$ in the first relation and use the fact that $P_n(x;a,b,c,d|q)$ is symmetric with respect to $a,b,c,d$, (cf. \cite[p.6]{Askey-1985}), to obtain the result.
\end{proof}
\begin{proposition} The structure relation (\ref{e11}) for monic Askey-Wilson polynomials is  
	\begin{align}\label{ee.1}
16abcd\left(x-\tfrac{a^{-1}+a}{ 2}\right)&\left(x-\tfrac{b^{-1}+b}{ 2}\right)\left(x-\tfrac{c^{-1}+c}{ 2}\right)\left(x-\tfrac{d^{-1}+d}{ 2}\right)D_q^2P_{n}(x;a,b,c,d|q)\nonumber\\
	&= \sum_{j=-2}^{2}a_{n,n+j}P_{n+j}(x;a,b,c,d|q),\; {\rm where}
	\end{align}
\vspace{-0.46 cm}
\begin{align*}		a_{n,n+2}=&~16abcd\gamma_{n}\gamma_{n-1},\\
		a_{n,n+1}=&~a_{n,n+2}\left(k_{n-2}^{(a,bq,cq,dq)} +k_{n-1}^{(b,a,cq,dq)} +k_n^{(c,b,a,dq)} +k_{n+1}^{(d,b,c,a)}\right),\\
			a_{n,n}=&~a_{n,n+2}\left[k_{n-2}^{(a,bq,cq,dq)}k_{n-2}^{(b,a,cq,dq)} +k_{n-1}^{(c,b,a,dq)}\left(k_{n-2}^{(a,bq,cq,dq)} +k_{n-1}^{(b,a,cq,dq)}\right)\right.\\ &\left. ~ +k_n^{(d,b,c,a)}\left(k_{n-2}^{(a,bq,cq,dq)}+k_{n-1}^{(b,a,cq,dq)} +k_n^{(c,b,a,dq)}\right)\right],\\
				a_{n,n-1}=&~a_{n,n+2}\left[k_{n-2}^{(a,bq,cq,dq)}k_{n-2}^{(b,a,cq,dq)}k_{n-2}^{(c,b,a,dq)}+ k_{n-1}^{(d,b,c,a)}k_{n-2}^{(a,bq,cq,dq)}k_{n-2}^{(b,a,cq,dq)} \right.\\& \left.+k_{n-1}^{(d,b,c,a)}k_{n-1}^{(c,b,a,dq)} \left(k_{n-2}^{(a,bq,cq,dq)}+k_{n-1}^{(b,a,cq,dq)}\right)\right],\\
				a_{n,n-2}=&~a_{n,n+2}\left(k_{n-2}^{(a,bq,cq,dq)} k_{n-2}^{(b,a,cq,dq)}k_{n-2}^{( c,b,a,dq)}k_{n-2}^{(d,b,c,a)}\right),
					\end{align*}
	and $k_n$ is given in Lemma \ref{lemma4}.
\end{proposition}
\begin{proof}
Using the fact that $\mathcal{D}_q^2P_n(x;a,b,c,d|q)=\gamma_{n}\gamma_{n-1}P_{n-2}(x;aq,bq,cq,dq|q)$ (cf. \cite[(14.1.9)]{KSL}) and taking into account the expression for the polynomial $\pi(x)$, given in Remark \ref{newremark}, (\ref{e11}) can be written as
\vspace{-0.1cm}
\begin{align}
\left(x-\tfrac{a^{-1}+a}{ 2}\right)&\left(x-\tfrac{b^{-1}+b}{ 2}\right)\left(x-\tfrac{c^{-1}+c}{ 2}\right)\left(x-\tfrac{d^{-1}+d}{ 2}\right)P_{n-2}(x;aq,bq,cq,dq|q)\nonumber\\
&=\sum_{j=-2}^{2}\frac{a_{n,n+j}}{16abcd\gamma_{n}\gamma_{n-1}}P_{n+j}(x;a,b,c,d|q).\label{ee5}
\end{align}
Replace $n$ by $n-2$, $b$ by $bq$, $c$ by $cq$ and $d$ by $dq$ in the  first equation of Lemma \ref{lemma4} to obtain

\begin{align}
	\lefteqn{\left(x-\tfrac{a^{-1}+a}{ 2}\right)P_{n-2}(x;aq,bq,cq,dq|q)}\nonumber\\
	&=P_{n-1}(x;a,bq,cq,dq|q)+k_{n-2}^{(a,bq,cq,dq)}P_{n-2}(x;a,bq,cq,dq|q).\label{na}\end{align}
Multiply \eqref{na} by $\left(x-\tfrac{b^{-1}+b}{ 2}\right)\left(x-\tfrac{c^{-1}+c}{ 2}\right)\left(x-\tfrac{d^{-1}+d}{ 2}\right)$ and use the other relations in Lemma \ref{lemma4} to transform (\ref{na}) into (\ref{ee5}) where the coefficients $a_{n,n+j}$, $j\in\{-2,...,2\}$ are written in terms of $k_{n+j}$.	\end{proof}
\begin{theorem}\label{prop1}
Let $x_{n,1}$ $ (x_{n,n})$ be the smallest (largest) zero of the Askey-Wilson polynomial $P_n(x;a,b,c,d|q)$. Then \begin{align}
	\label {up1}x_{1,n}&<\frac { 2(\,{q}^{n-1}+1)  \left( {q}^{n-1} \left( aA+C \right) -a-B \right)  \left( aC{q}^{n-1}-1 \right) -
			\sqrt {{\it  I_n}} }{ 8\left( \,aC{q}^{2\,n-2}-1
			\right)  \left( aC{q}^{n-1}-1 \right) }\\
	\label{lb1}	x_{n,n}&>\frac { 2(\,{q}^{n-1}+1)  \left( {q}^{n-1} \left( aA+C \right) -a-B \right)  \left( aC{q}^{n-1}-1 \right) +
			\sqrt {{\it I_n} }}{ 8\left( \,aC{q}^{2\,n-2}-1
			\right)  \left( aC{q}^{n-1}-1 \right) }
		\end{align} where $A=bc+bd+cd, \quad B=b+c+d, \quad C=bcd$ and
\begin{align*}
I_n
&=-16 \left( aC{
	q}^{2\,n-2}-1\right)  \left( aC{q}^{n-1}-1 \right)\, \left[\left( -{q}^{3\,n-3}aC-1 \right)  \left( aC-aB-
A+1 \right) \right.\\&+(( {C}^{2}+{b}^{2}{c
}^{2}+{b}^{2}{d}^{2}+{c}^{2}d^2+bcdB-A) {a}^{2}+ A\left( C-B
\right) a+C^2-CB) {q}^{2n-2}\\
&\left.+( (1-A) {a}^{2}- \left( A-1
\right) B a-CB+{b}^{2}
+A+{c}^{2}+{d}^{2}+1 ) {q}^{n-1}\right]\\ 
& +4\left( {q}^
{n-1}+1 \right) ^{2} \left({q}^{n-1}aA+{q}
^{n-1}C-a-B \right) ^{2} \left( aC{q}^{n-1}-1 \right) ^{2}.
\end{align*}

\end{theorem}
\begin{proof}  Use the three-term relation \cite[(14.1.5)]{KSL} to transform (\ref{ee.1}) into
\begin{align*}
&\left(x-\tfrac{a^{-1}+a}{ 2}\right)\left(x-\tfrac{b^{-1}+b}{ 2}\right)\left(x-\tfrac{c^{-1}+c}{ 2}\right)\left(x-\tfrac{d^{-1}+d}{ 2}\right)P_{n-2}(x;aq,bq,cq,dq|q)\\
&={\frac {(1-q)  \left( dbca \left( {q}^{n} \right)
		^{2}-q \right) }{4a\sqrt {q} \left( -q+{q}^{n} \right)  \left( -1+{q}^{
			n} \right) cdb}}\psi(x)P_{n+1}(x;a,b,c,d|q)+G_{2,n}(x)P_{n}(x;a,b,c,d|q)
\end{align*} where $\psi$ is the polynomial coefficient of (\ref{e10}) given in \eqref{psi}
and	\begin{align*}
	&\frac{4abcd(q^n;q^{-1})_2(abcdq^{2n}-1)}{abcdq^{n-1}-1}
	G_{2,n}(x)=4(abcd {q}^{2n}-1)(abcdq^n-1)x^2\\
	&-(2q^n+2)(q^n (abc+abd+acd+qbcd)-a-b-c-d)(abcdq^{n-1}-1)x\\
	&-(q^{3n}abcd+1)(dbca-ab-ac-ad-bc-bd-cd+1) + (( {b}^{2}{c}^{2}{d}^{2}+{b}^{2}{c}^{2}+{b}^{2}cd\\
	&+b^2d^2+
	bc^2d+bcd^2+c^2d^2-bc-bd-cd)a^2+(bc
	+bd+cd)(dbc-b-c-d)a\\
	&+bdc(dbc-b-c-d))q^{2n}+((1 -bc-bd-cd)a^2-(bc+bd+cd-1)(d+c+b)a\\
	&-b^2cd-bc^2d-bcd^2+b^2+bc+bd+c^2+cd+d^2+1)q^n.
	\end{align*}It follows from \cite[Cor. 2.2]{DriverJordaan2012} that the zeros of the second degree polynomial $G_{2,n-1}$ yield inner bounds for the extreme zeros of $P_n(x;a,b,c,d|q)$ and the result follows.
\end{proof}
Bounds for the zeros of Askey-Wilson polynomials obtained in Theorem \ref{prop1} for some special values of the parameters $n$, $a$, $b$, $c$, $d$ and $q$ are illustrated in Table 1.
\begin{table}[!h]\caption{Zeros of monic Askey-Wilson polynomials for $n=7,9,12$ respectively and $(a,b,c,d,q)=(\tfrac{6}{7}, \tfrac{5}{7},  \tfrac{4}{7}, \tfrac{3}{7}, \tfrac{1}{9})$}
	\begin{tabular}{ |c|c|c|c|c| }
		\hline
		Value of n & 7 & 9&12\\ 
		\hline
		Smallest zeros of $P_n(x;a,b,c,d|q)$ &-0.864348856 & -0.922505234& -0.95879261 \\ 
		\hline
		Upper bound (\ref{up1})& 0.33690627 & 0.336904827& 0.336904809\\ 
		\hline
		Lower bound (\ref{lb1})&0.948809497&0.948809477&0.948809477\\
		\hline
		Largest zeros of  $P_n(x;a,b,c,d|q)$&0.981913401&0.986122226&0.990012586\\	
		\hline
	\end{tabular}
\end{table} 

\medskip Special cases of Askey-Wilson polynomials arise when one or more of the parameters vanish and bounds for the extreme zeros of these special cases, namely continuous dual $q$-Hahn, Al-Salam Chihara, continuous big $q$-Hermite and continuous $q$-Hermite polynomials, can be deduced from the bounds in Theorem \ref{prop1}.
\section*{Acknowledgments}
The authors thank Mourad Ismail and Tom Koornwinder for helpful discussions and comments and the anonymous referee for constructive comments that resulted in substantial improvements to the paper.
 \bibliographystyle{amsplain}
 
\end{document}